\documentclass[a4wide,11pt]{amsart}
\usepackage{amsmath, amsthm, amscd, amssymb}
\usepackage{tikz,cite}
\usepackage{caption,subcaption}
\usepackage{hyperref}
\newtheorem{thm}{Theorem}[section]
\newtheorem{prop}[thm]{Proposition}
\newtheorem{clm}[thm]{Claim}

\newtheorem{lem}[thm]{Lemma}
\theoremstyle{definition}
\newtheorem{problem}[thm]{Problem}

\newtheorem{rmk}[thm]{Remark}

\numberwithin{equation}{section}
\newcommand{\Z}{\mathbb{Z}}
\newcommand{\s}{\mathrm{s}}
\newcommand{\cs}{\mathrm{cs}}
\newcommand{\G}{\mathcal{G}^{\mathrm{cs}}}
\newcommand{\W}{\mathcal{W}^{\mathrm{cs}}}

\newenvironment{alg-enumerate}{%
  \begin{enumerate}%
 }{\end{enumerate}%
}

\begin{document}
\title{On {the} weighted safe set problem on paths and cycles}

\author[S. Fujita]{ Shinya Fujita}
\address[S. Fujita]{School of Data Science, Yokohama City University, Yokohama 236-0027, Japan}
\email{fujita@yokohama-cu.ac.jp}

\author[T. Jensen]{Tommy Jensen}
\address[T. Jensen]{Institute of Mathematics, Aarhus Universitet, 8000 Aarhus, Denmark}
\email{au611354@mit.au.dk}

\author[B. Park]{Boram Park}
\address[B. Park]{Department of mathematics, Ajou University, Suwon, 443-749, Republic of Korea.}
\email{borampark@ajou.ac.kr}

\author[T. Sakuma]{Tadashi Sakuma}
\address[T. Sakuma]{Faculty of Science, Yamagata University, Yamagata 990-8560, Japan}
\email{sakuma@sci.kj.yamagata-u.ac.jp}

\thanks{Fujita's work was supported by JSPS KAKENHI (No.~15K04979). Park's work was supported by Basic Science Research Program through the National Research Foundation of Korea (NRF) funded by the Ministry of  Science, ICT \& Future Planning (NRF-2018R1C1B6003577).
Sakuma's work was supported by JSPS
KAKENHI (No.~26400185, No.~16K05260 and No.~18K03388).}
\keywords{Safe set; Connected safe set; Weighted graph; Safe-finite; Subgraph component polynomial}

\date{\today}
\maketitle

\begin{abstract}

Let $G$ be a graph, and let {$w$ be a positive real-valued weight function on $V(G)$.} For every subset $X$ of $V(G)$, let $w(X)=\sum_{v \in X} w(v).$ A non-empty subset $S \subset V(G)$ is a {\it weighted safe set} of $(G,w)$ if, for every component $C$ of the subgraph induced by $S$ and every component $D$ of $G-S$, we have $w(C) \geq w(D)$ whenever there is an edge between $C$ and $D$.
If the subgraph of $G$ induced by a weighted safe set $S$ is connected, then the set $S$ is called a {\it connected weighted safe set} of $(G,w)$. The \textit{weighted safe number} $\s(G,w)$ and \textit{connected weighted safe number} $\cs(G,w)$ of $(G,w)$ are the minimum weights $w(S)$ among all weighted safe sets and all connected weighted safe sets of $(G,w)$, respectively. It is easy to see that for any pair $(G,w)$, $\s(G,w) \le \cs(G,w)$ by their definitions. In this paper, we discuss the possible equality when $G$ is a path or a cycle. We also give an answer to a problem due to Tittmann et al.\ [Eur. J. Combin. Vol. 32 (2011)] concerning subgraph component polynomials for cycles and complete graphs.
\end{abstract}

%%%%%%%%%%%%%%%%%%
\section{Introduction}
%%%%%%%%%%%%%%%%%%

We start with a question about number sequences in combinatorial number theory.
For a sequence $a_1, \ldots, a_n$ of positive integers and a segment $I$
consisting
of a
subsequence $a_i, a_{i+1},\ldots, a_{i+|I|-1}$, let $s(I)=\sum_{j=i}^{j=i+|I|-1}a_j$.
We consider partitioning $a_1,\ldots ,a_n$ into an odd number of non-empty
segments $I_1=\{a_1,\ldots, a_{|I_1|}\},\ldots, I_{2k+1}=\{a_{n-|I_{2k+1}|+1},
\ldots, a_n\}$ with $k\ge 1$ so that {the sequence of $s(I_1)$, $s(I_2)$, $\ldots$, $s(I_{2k+1})$ is a ``zigzag" sequence, i.e.,}
$\max\{s(I_{2j-1}), s(I_{2j+1})\}\le s(I_{2j})$ holds for all $j=1,\ldots,k$.
Whenever such segments exist,
we would like to choose them so that $\sum_{j=1}^{k}s(I_{2j})$ is as
small as possible and, subject to this condition, $k$ is as small as possible
among all such partitions into an odd number of segments.
{By our choice, $k$ is likely to be small in many cases.}
Assuming that the desired partition exists,
we consider the special case when {$n$ is odd and the optimal solution occurs only when each segment {consists of a single number.}}
Then the elements of the odd
segments $I_{2j-1}$ for $j=1\ldots,k+1$ and the elements of the
even segments $I_{2j}$ for $j=1,\ldots,k$ correspond to the two number
sequences that are
obtained from $a_1,\ldots, a_n$ by taking their terms alternately.
The question is whether number sequences $a_1,\ldots ,a_n$ exist for
 {which $k=\frac{n-1}{2}$ is the optimal solution.}

Our answer to this question is positive. Actually, we construct infinitely
many number sequences with this property (see   Proposition~\ref{path}).
Along a slightly different line, we can also ask a similar question for
cyclic
number sequences
$a_0,a_1,\ldots , a_{n-1},$ with the indices taken modulo $n$
(that is, $a_n=a_0$)
by modifying the problem into finding an even number of segments of subsequences
$I_1,\ldots, I_{2k}$ subject to the same requirements, where $k\ge 1$.
Our answer to the modified question is rather negative.
Indeed, we show that $k=1$ is optimal for any cyclic number sequence
(see Theorem~\ref{thm:cycle}).

In fact the above problems are related to safe set problems on weighted graphs.
We use \cite{cl} for graph terminology and notation not defined here.
Only finite, simple (undirected) graphs are considered. For a graph $G$, let
$\delta(G)$ be the minimum degree of $G,$ and let $G[S]$ denote the subgraph
of $G$ induced by {a} subset
$S \subset V(G).$ We often abuse/identify terminology and notation for
subsets of the vertex set and subgraphs induced by them. In particular, a
(connected) component is sometimes treated as a subset of the vertex set.
We let $k(G)$ denote the number of components of $G$.
When $A$ and $B$ are disjoint subsets of $V(G)$, the set of edges joining
a vertex of $A$ to a vertex of $B$ is denoted by $E_G(A,B)$. If there
is no confusion, we often denote this set by $E(A,B)$.
If $E(A,B)\neq\emptyset$, then $A$ and $B$ are said to be \textit{adjacent}.

A \textit{weight function} $w$ on $V(G)$ is a mapping associating
each vertex of $G$ with a positive real number.
Let $\mathcal{W}(G)$ be the set of all weight functions on $V(G)$.
For $w\in \mathcal{W}(G)$, we refer to $(G,w)$ as a \textit{weighted graph}.
For every subset $X$ of $V(G)$, let
$w(X)=\sum_{v \in X} w(v);$ here we also allow ourselves to use the notation $w(G[X])$ for $w(X)$.

The notion of a safe set was introduced by Fujita et al.\ \cite{SGS-safe-set}
{as} a variation of facility location problems. Bapat et al.\
\cite{BFLMMST-weighted-sf} extended it to weighted graphs.
Assume that $(G,w)$ is a weighted graph where $G$ is connected.
A non-empty subset $S \subset V(G)$ is a {\it weighted safe set} of $(G,w)$ if
for every component $C$ of $G[S]$ and every component $D$ of $G-S$, we have
$w(C) \geq w(D)$ whenever $E (C,D) \neq \emptyset$.
The \textit{weighted safe number} of $(G,w)$ is the minimum weight $w(S)$
among all weighted safe sets of $(G,w)$, that is,
$$\s(G,w)=\min\{ w(S) \mid S \text{ is a weighted safe set  of }(G,w)\}.$$
If $S$ is a weighted safe set of $(G,w)$ and $w(S)=\s(G,w)$,
then $S$ is called a {\it minimum weighted safe set\/}.
Restricting to connected safe sets, if  $S$ is a weighted safe set of
$(G,w)$ and $G[S]$ is connected, then $S$ is called a \textit{connected
weighted safe set} of $(G,w)$.
The \textit{connected weighted safe number} of $(G,w)$ is defined by
$$\cs(G,w)=\min\{ w(S) \mid S \text{ is a connected weighted safe set  of }(G,w)\},$$
and  a {\it minimum connected weighted safe set\/} is
a connected weighted safe set $S$ of $(G,w)$ such that $w(S)=\cs(G,w)$.
Throughout the paper, we will often omit `weighted', and simply speak of a safe set
or a connected safe set.

{For a disconnected graph $G$, we can define the notion of a (connected) safe set naturally by considering a (connected) safe set of each component. So, we always assume that every graph in this paper is connected unless otherwise specified.}

Recently, problems on safe sets in graphs have been extensively studied,
especially to investigate the algorithmic aspects.
Fujita et al.\ \cite{SGS-safe-set} showed that computing the connected
safe number in { a unweighted graph (i.e., $(G,w)$ with a constant weight function $w$)} is NP-hard in general,
whereas they constructed a linear time algorithm for computing the connected
safe number in unweighted trees. \'Agueda et al.\ \cite{ag} gave an efficient
algorithm for computing the safe number for unweighted trees.
Somewhat surprisingly, Bapat et al.\ \cite{BFLMMST-weighted-sf} showed that
computing the connected weighted safe number in a tree is NP-hard even if
the underlying tree is restricted to be a star, whereas they gave an
efficient algorithm computing the safe number for a weighted path.
More recently, Ehard and Rautenbach  \cite{ED-tree} provided
a polynomial-time approximation scheme (PTAS) for the connected safe number
of a weighted tree.

In this paper we focus on weighted graphs $(G,w)$ with  $\s(G,w) = \cs(G,w)$.
{Recall that, for any weighted graph $(G,w)$, we have
$\s(G,w) \le \cs(G,w) < 2\s(G,w)$, where the fist inequality is from definitions and the second inequality is obtained from the same method as in Proposition 2 of \cite{SGS-safe-set}.}
Intuitively, {like for the facility location problem}, if $(G,w)$ contains a minimum connected weighted safe set $S$
with  $|S|=\s(G,w)$, then one might feel that $G[S]$ plays a central role
in the graph. To see that this is reasonable, 
{consider a weighted graph $(G,w)$.}
When we regard $(G,w)$ as a kind of network, $G[S]$ has a majority weight
compared with other components in $G-S$ and the internal structure of $G[S]$
is rather stable because it is connected; thus, we can regard $G[S]$ as a
core part in the network in some sense.
From this viewpoint, if $G$ is
a graph such that $\s(G,w) = \cs(G,w)$ holds for a weight
$w\in \mathcal{W}(G)$, then we would choose a minimum safe set $S$ which induces a connected graph for efficiency and stableness.
Consequently we would like to investigate which kind of weighted graphs $(G,w)$
satisfy $\s(G,w) = \cs(G,w)$. We thus propose the following problems.

\begin{problem}\label{problem1:weights}
Given a graph $G$, determine the set $\W(G)$ of weights
$w$ such that $\s(G,w)=\cs(G,w)$.
\end{problem}

\begin{problem}\label{problem2:graphs}
Determine  the family $\G$   of graphs $G$ for which $\s(G,w)=\cs(G,w)$ for
every $w\in \mathcal{W}(G)$.
\end{problem}

If $G$ is a complete graph, then $\W(G)=\mathcal{W}(G)$ and hence $G\in \G$.
If $G$ is a path, then it is shown in \cite{SGS-safe-set} that any constant
function belongs to $\W(G)$.

Regarding Problem~\ref{problem2:graphs}, in addition to every complete graph
being in $\G,$ it is not difficult to check that any star graph is in $\G$;
indeed, we can even show the following.

\begin{prop}\label{n-1:vertex}
If $\Delta(G)=|V(G)|-1$, then $G\in \G$.
\end{prop}

\begin{proof}[Proof of Proposition~\ref{n-1:vertex}]
Let $v$ be a vertex of degree $|V(G)|-1$ in $G$.
Let $w$ be a weight function on $V(G)$, and $S$ be a  minimum safe set of $(G,w)$.
Suppose that $G[S]$ is not connected. Then $v\not\in S,$ hence
$G-S$ contains $v$ and so is connected.
Let $D$ be a component
of $G[S]$
with the smallest weight, that is,
\[ w(D) =\min \{ w(C) \mid C\text{ is a component of }G[S] \}.\]
Let $S'=V(G)-D,$ then $G[S']$ is connected, since $v\in S'$.
Since $S'$ includes a component of $G[S]$ whose weight is not less than $w(D)$,
it follows that $w(D) \le w(S')$.
Thus $S'$ is a connected safe set of $(G,w)$.
Moreover, since $w(D) \ge w(G-S)$, $w(S')\le w(S)$. Thus  $w(S')= w(S)$.
Hence, $\s(G,w)=\cs(G,w)$.
\end{proof}

In view of Proposition~\ref{n-1:vertex}, one might ask if there exists
a graph $G$ with a low maximum degree such that $G\in \G$.
As we will observe later, this is not the case for paths.
In the following we show that such graphs do exist.
Namely, we prove the following theorem.

\begin{thm}\label{thm:cycle}
Any cycle belongs to $\G$.
\end{thm}

 In fact, cycles are the unique non-complete graphs for which a minimum
safe set always contains at least half the total weight
of the graph.
Through the study of the weighted safe set problem for cycles, we found
several equivalent conditions for a graph to be a cycle or a complete graph.
In particular, one of them sheds light on the study of subgraph component
polynomials. Inspired by the study of {the} community structure in connection
networks, Tittmann et al.\ \cite{TAM} introduced a new type of graph
polynomial.  For a graph $G$ and two positive integers $i$ and $j$, let
\[q_{i,j}(G)=| \{X \subset V(G) \mid |X|=i \text{ and }k(G[X])=j \}|.\]
The \textit{subgraph component  polynomial} $Q(G;x,y)$ of $G$ is the
polynomial in two variables $x$ and $y$ such that the coefficient of
$x^i y^j$ is $q_{i,j}(G)$.
From the definition, it is easy to check that  $q_{1,1}(G)=|V(G)|$,
$q_{2,1}(G)=|E(G)|$, and $q_{2,2}(G)={n\choose 2}-|E(G)|$. {In addition, $q_{1,1}(G)=q_{n-1,1}(G)$  is equivalent to the statement that $G$ is 2-connected.}
Our result is the following.

\begin{thm}\label{prop:ratio}
Let $G$ be a connected graph with $n$ vertices, where $n\ge 5$.
The following are equivalent:
\begin{itemize}
\item[\rm(i)] $G$ is either a complete graph or a cycle;
\item[\rm(ii)] $\s(G,w)\ge \frac{w(G)}{2}$ for every $w\in \mathcal{W}(G)$;
\item[\rm(iii)] $G-\{u,v\}$ is disconnected for any two nonadjacent vertices
$u$ and $v$;
\item[\rm(iv)] $q_{1,1}(G)=q_{n-1,1}(G)$  and
$q_{2,1}(G)=q_{n-2,1}(G)$. 
\end{itemize}
\end{thm}

Tittmann et al.\ gave a few examples of graphs and graph families that
are determined by $Q(G;x,y)$ and, in view of the importance of the
study of {the} communication structure in networks, they proposed as an open
problem to find more classes of graphs that are determined by $Q(G;x,y)$
(see Problem 35 in \cite{TAM}). Our result
contributes to a solution of their open problem.

Our result suggests a deep relationship between safe numbers and
subgraph component polynomials.
Consequently we believe that Problem~\ref{problem2:graphs} is also
important in the study of communication structure in networks. As an
initial step to approach this challenging problem we consider some
basic properties of $\G$ from a variety of viewpoints.

For a vertex $v$ of degree two in $G$ which is not on a triangle,
\textit{suppression} of $v$ is the operation of removing $v$ and adding an
edge between the two neighbors of $v$.
This is the reverse operation of \textit{subdivision}; the
subdivision of an edge $e=xy$ yields a new graph containing one
new vertex $v,$ and having an edge set in which $e$ is replaced by two new edges
$xv$ and $vy$.
  We have the following result.

\begin{thm}\label{prop:subdivision} The family $\G$ is closed under suppression.
\end{thm}

In contrast to the family $\G$ of graphs, one might consider another
family of graphs that is in a certain sense very far from $\G.$
A family $\mathcal{G}$ of graphs is \textit{safe-finite} if
$f(\mathcal{G}) < \infty,$ where
\begin{center}
$f(\mathcal{G}) = \max_{G\in \mathcal{G}, w\in \mathcal{W}(G) }
\min \{k(G[S])\mid S \text{ is a minimum weighted safe set of } (G,w)\}.$
\end{center}
Conversely, $\mathcal{G}$ is \textit{safe-infinite} if it is not safe-finite.
Obviously any finite family $\mathcal{G}$ of graphs is safe-finite.
The function $f$ is a mapping from a safe-finite family
of graphs to a positive integer and, in particular,
 $\G$ is the maximal family $\mathcal{G}$ of graphs such that $f(\mathcal{G})=1.$
It would seem an interesting problem to discover what kind of families
of graphs are safe-(in)finite.
Considering the set of all complete bipartite graphs with a constant weight on
the vertices, we observe that there exists a safe-infinite family of graphs.
To present another example, we provide the following theorem.

\begin{thm}\label{oddpath}
 The family of paths with an odd number of vertices is safe-infinite.
\end{thm}

This paper is organized as follows. In Section~\ref{sec:path},we give a construction on weighted paths $(P,w)$ {such that every minimum safe set has at least $N$ components for an arbitrary taken positive integer $N$,} 
thereby proving Theorem~\ref{oddpath}. We also prove Theorem~\ref{prop:subdivision}
in this section. We prove Theorems~\ref{thm:cycle} and~\ref{prop:ratio} in
Sections~\ref{sec:cycle} and~\ref{new}, respectively.
We give some remarks on Theorems~\ref{thm:cycle} and~\ref{prop:ratio} in
Section~\ref{sec:open}.

\section{Weighted safe sets of paths}\label{sec:path}
Subsection~\ref{subsec:path} gives the construction of a
weighted path $(P,w)$ with {an odd number of vertices} such that any minimum safe set has exactly
$\lfloor|V(P)|/2\rfloor$ components. This implies Theorem 1.7.
Further, {in Subsection~\ref{subsec:subdivision},} we give a construction,
using  subdivisions, of weight functions $w$ of $P$ for any path $P$ so that $w\not\in \W(P)$.

\subsection{Construction of weight functions on a  path of odd order}\label{subsec:path}
Throughout the subsection, let $n\ge 2$ be an integer, $P: v_1v_2\ldots v_{2n+1}$ be a path with $2n+1$ vertices.
Fix two positive real numbers $a$ and $b$ so that  \begin{eqnarray}\label{condition:ab}
2b>3a\qquad\text{ and }\qquad 2a>b>a.
\end{eqnarray}
We define a weight function  $w$ on $V(P)$ by (See Figure~\ref{fig:path}.)
 \begin{eqnarray}\label{def:weight}
 w(v_j)=\begin{cases} b & \text{ if }j\in\{1,2\}  \\
  2^{i-1}a  &\text{ if }j\ge 3\text{ and }j\in\{2i,2i+1\}\text{ for some }i\in\{1, \ldots,n\}.\\
  \end{cases}\end{eqnarray}

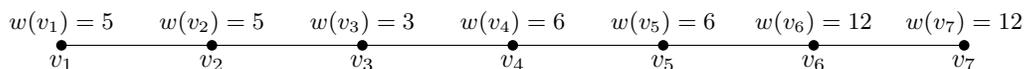
\begin{figure}[h!] \centering
\begin{tikzpicture}
    \path
        (0,0) coordinate (1)
        (2,0) coordinate (2)
        (4,0) coordinate (3)
        (6,0) coordinate (4)
        (8,0) coordinate (5)
        (10,0) coordinate (6)
        (12,0) coordinate (7);
\fill   (1) node[above]{\footnotesize$w(v_1)=5$}   (2) node[above]{\footnotesize$w(v_2)=5$}
(3) node[above]{\footnotesize$w(v_3)=3$}   (4) node[above]{\footnotesize$w(v_4)=6$}
(5) node[above]{\footnotesize$w(v_5)=6$}   (6) node[above]{\footnotesize$w(v_6)=12$}
(7) node[above]{\footnotesize$w(v_7)=12$};	
\fill   (1) node[below]{\small$v_1$}   (2) node[below]{\small$v_2$}
(3) node[below]{\small$v_3$}   (4) node[below]{\small$v_4$}
(5) node[below]{\small$v_5$}   (6) node[below]{\small$v_6$}
(7) node[below]{\small$v_7$};
	\path (1) edge (2);
	\path (2) edge (3);
	\path (3) edge (4);
	\path (4) edge (5);
	\path (5) edge (6);
	\path (6) edge (7);
    \fill
        (1) circle (2pt)
        (2) circle (2pt)
        (3) circle (2pt)
        (4) circle (2pt)
        (5) circle (2pt)
 (6) circle (2pt)
               (7) circle (2pt);
	\end{tikzpicture}
\caption{$(P,w)$ when $n=3$ and $a=3,b=5.$}\label{fig:path}
\end{figure}

We let $S=\{v_2,v_4,\ldots, v_{2n}\}$.
Since $w(v_{1})=w(v_{2}) >w(v_{3})$  and $w(v_{2i-1})<w(v_{2i})=w(v_{2i+1})$
for all $i\in\{2, \ldots,n\}$, it is easy to see that $S$ is a safe set of $(P,w)$.
Actually, $S$ is the unique minimum  safe set.

\begin{prop}\label{path}
Let $(P,w )$ be a  weighted path with a  weight function defined as
in \eqref{def:weight}.
Then $S=\{v_2,v_4,\ldots, v_{2n}\}$ is the unique minimum safe set of $(P,w )$.
\end{prop}

\begin{proof}[Proof of Proposition~\ref{path}]
Take  a  minimum safe set $X$ of  $(P,w )$.
Since $S$ is a  safe set and $X$ is a minimum  safe set of $(P,w)$,
together with \eqref{condition:ab},
\begin{eqnarray}
 \s(P,w)= w(X) \le w(S)=\sum_{i=1}^{n}w(v_{2i}) =
b+ \sum_{i=2}^{n} 2^{i-1}a = 2^na-2a+b< 2^{n}a.  \label{last_two_2-0}
\end{eqnarray}

\begin{clm}\label{p2} $|X\cap\{ v_{2i}, v_{2i+1} \}|=1$ for all
$i\in \{2, \ldots,n\}$.
\end{clm}
\begin{proof}[Proof of Claim~\ref{p2}]
We apply induction on $n-i \geq 0,$ where $n$ is fixed.
Since  $w(\{v_{2n}, v_{2n+1}\})=2^na$,
if $v_{2n}$ and $v_{2n+1}$  are in a same component of either $P[X]$ or $P-X$,
then this component has weight at least $2^na$, contradicting \eqref{last_two_2-0}.
Hence, $|X\cap\{ v_{2n}, v_{2n+1} \}|=1$.
Assume $|X\cap\{ v_{2i'}, v_{2i'+1} \}|=1$ for all $i'>i$ ($2\le  i\le n-1$).
Then
\[w(X\cap\{ v_{2i+2},v_{2i+3}, \ldots,v_{2n+1}\}) = \sum_{i'=i}^{n-1} 2^{i'}a =
a(2^{i} +2^{i+1}+\cdots+ 2^{n-2} +
 2^{n-1}) = a(2^n   -2^i).\]
If $X\cap\{ v_{2i}, v_{2i+1} \}= \{ v_{2i}, v_{2i+1} \}$, then
$w(X)\ge 2^n a -2^ia +w(\{v_{2i}, v_{2i+1}\})=2^na$, contradicting \eqref{last_two_2-0}.
Suppose $X\cap\{ v_{2i}, v_{2i+1} \}= \emptyset,$ and let
$D$ be the component of $P-X$ that contains $v_{2i}$ and $v_{2i+1}.$
Then $w(D)\ge  2^i a$.
If  there is a component $C'$ of $P[X]$ adjacent to $D$ such that
$C'\subset \{v_1,v_2,\ldots,v_{2i-1}\},$
then
\[w(X)\ge 2^n a -2^ia +w(C') \ge  2^n a -2^ia +w(D)\ge  2^n a,\]
a contradiction to \eqref{last_two_2-0}.
Suppose that there is no component of $P[X]$ contained in $\{v_1,v_2,\ldots,v_{2i-1}\}$.
Then $D \supset \{v_1,v_2,\ldots,v_{2i+1}\},$
and there is a unique component $C$ of $P[X]$ which is adjacent to $D$.
By the induction hypothesis,  $C=\{v_{2i+2}\}$ or $C=\{v_{2i+3}\}$ or $C=\{v_{2i+3},v_{2i+4}\}$.
If $C=\{v_{2i+2}\}$ or $C=\{v_{2i+3}\}$, then, together with~\eqref{condition:ab},
\[  w(D) \ge \sum_{i'=1}^{2i+1} w(v_{i'}) =  2^{i+1}a-3a+2b > 2^{i}a=w(C),\]
a contradiction to the definition of a  safe set.
Similarly, if  $C=\{v_{2i+3},v_{2i+4}\}$, then  $D= \{v_1,v_2,\ldots,v_{2i+2}\},$ and so
\[  w(D) = \sum_{i'=1}^{2i+2} w(v_{i'}) = 2^{i+1}a-3a+2b +2^ia > 2^{i}a+2^{i+1}a\ge w(C),\]
a contradiction to  the definition of a  safe set.
Hence $|X\cap\{ v_{2i}, v_{2i+1} \}|=1.$
\end{proof}

By Claim~\ref{p2}, $w(X\cap\{v_4,v_5,\ldots,v_{2n+1}\}) =2^na-2a$.
Together with \eqref{last_two_2-0}, it follows that $|X\cap\{v_1,v_2,v_3\}|\le 1$.
Furthermore, the following holds.

\begin{clm}\label{p3} $X\cap\{v_1,v_2,v_3,v_4,v_5\}=\{v_2,v_4\}$.
\end{clm}
\begin{proof}[Proof of Claim~\ref{p3}]
Suppose $v_4\not\in X,$ and
let $D$ be the component of $P-X$ containing $v_4$.
If $X\cap\{v_1,v_2,v_3\}\neq \emptyset$, then $X\cap\{v_1,v_2,v_3\}=\{v_i\}$
is a component of $P[X]$, which is a contradiction since $w(D)\ge w(v_4) > w(v_i)$.
Thus  $X\cap\{v_1,v_2,v_3\}= \emptyset$ and so $D=\{v_1,v_2,v_3,v_4\}$.
Note that by Claim~\ref{p2}, $v_5\in X$, and
the component of $P[X]$ containing $v_5$ is  either $ \{v_5\}$ or $ \{v_5,v_6\}$.
However, by \eqref{condition:ab}, $w(\{v_5,v_6\}) \le 2a+4a < 3a+2b =w(D)$.
Hence, $v_4\in X$  and by Claim~\ref{p2}, $v_5\not\in X$.

By the fact that $|X\cap\{v_1,v_2,v_3\}|\le 1$, it remains to show $v_2\in X$.
Suppose not, and let $C$ be the component of $P[X]$ containing $v_4$.
Then $C$ is either $\{v_3,v_4\}$ or $\{v_4\}$.
If $C= \{v_3,v_4\}$, then  $w(v_3)+w(v_4)>w(v_1)+w(v_2)$ by the definition of
a safe set, equivalently, $a+2a\ge b+b$, a contradiction to \eqref{condition:ab}.
If $C=\{v_4\}$, then
$w(v_4)\ge w(v_2)+w(v_3)$ by the definition of a safe set, equivalently,
$2a\ge b+a$, a contradiction to \eqref{condition:ab}.
Hence, $v_2\in X$. Consequently, the claim holds.
\end{proof}

Using induction on $i$ we will show for each $i\in \{1,2, \ldots,n\}$
that $\{v_{2i}\}$ is a component of $P[X].$
By Claim~\ref{p3}, each of $\{v_2\}$ and $\{v_4\}$ is a component of $P[X],$
so assume $i \ge 3.$
By the induction hypothesis, $\{v_{2i-2}\}$ is a component of $P[X]$.
If $v_{2i}\not\in X$ and $v_{2i+1}\in X,$
then $\{v_{2i-1},v_{2i}\}$ is a component of $P-X$ the weight of which
is greater than $w(v_{2i-2}).$ But this is a contradiction to the definition
of a  safe set.
Hence $v_{2i}\in X$ and $v_{2i+1}\not\in X$ follow by Claim~\ref{p2},
and the proposition holds.
 \end{proof}

{At the end of this subsection, we will show that, if a weight function $w$
on a path $P$ with $n$ vertices is bounded, then both the
safe number and the  connected safe number tend to
$\frac{w(P)}{3}$ as $n$ goes to $\infty$.
\begin{lem}\label{path-k-2k+1}
Let $(P,w)$ be a weighted path, and $S$ be a safe set of $(P,w)$. Then
$\frac{k}{2k+1}\leq\frac{w(S)}{w(P)}$, where $k$ is the number of  components of $P[S]$.
\end{lem}
\begin{proof}Let $S_1,\ldots,S_k$ be the components
of $P[S]$, and $P-S=D_1\cup \cdots \cup D_{k+1}$ where
the left (resp. right) neighbor of
$S_i$ is $D_i$ (resp. $D_{i+1}$) for all $i \in \{1,\ldots,k\}$. Note that
for every $i\not\in\{1,k+1\}$, $D_i$ is a component of $G-S$ and each of $D_1$ and $D_{k+1}$ is either a component of $P-S$ or empty.
Take an integer $r\in \{1,\ldots,k\}$ such that
$w(S_r)=\min\{w(S_i) : i=1,\ldots,k\}$.
By the definition of a safe set,
for each $i\in\{1,\ldots,r\}$, $w(S_i)\geq w(D_i)$
and for each $i\in\{r,\ldots,k\}$, $w(S_{i})\ge w(D_{i+1})$.
Then \[w(S)+w(S_r)=\sum_{i=1}^{r}w(S_i) +\sum_{i=r}^{k}w(S_i) \ge
\sum_{i=1}^{r} w(D_i) +\sum_{i=r}^{k}w(D_{i+1}) = w(P)-w(S),\]
and hence $2w(S)+w(S_r)\ge w(P)$.
Since $w(S_r)\le \frac{w(S)}{k}$, we have $\frac{k}{2k+1} \leq \frac{w(S)}{w(P)}$.
\end{proof}
\begin{prop}
Let $a$ and $b$ be real
numbers with $0<a < b$.
Let $\{ (P_n,w_n) \}_{n=1}^{\infty}$ be a sequence of weighted paths $(P_n,w_n)$ where $P_n$ is a path with $n$ vertices and $a \le w_n(v)\le b$ for every vertex $v\in V(P_n)$.
Then $\lim_{n \to \infty}\frac{\s(P_n,w_n)}{w_n(P_n)}=\lim_{n \to \infty}\frac{\cs(P_n,w_n)}{w_n(P_n)}=\frac{1}{3}.$
\end{prop}
\begin{proof}
Let $n$ be any positive integer.
We can find a subpath $L_n$ of $P_n$
starting from one pendent vertex of $P_n$ such that
$\frac{1}{3}w(P_n) - b \leq w(L_n) \leq \frac{1}{3}  w(P_n)$ holds.
By symmetry, we can also find a subpath $R_n$ of $P_n$
starting from the other pendent vertex such that
$\frac{1}{3}w(P_n) - b \leq w(R_n) \leq \frac{1}{3}  w(P_n)$ holds.
Then $P_n-(L_n\cup R_n)$ is a connected safe set of $(P_n,w_n)$ and so
$\cs(P_n,w_n) \leq  \frac{1}{3} w(P_n) +2b$.
Hence
$\frac{\cs(P_n,w_n)}{w_n(P_n)} \leq \frac{1}{3} + \frac{2b}{w(P_n)}
\leq \frac{1}{3} + \frac{2b}{an}$.
Together with Lemma~\ref{path-k-2k+1}, we have
\[ \frac{1}{3} \leq \frac{\s(P_n,w_n)}{w_n(P_n)} \leq \frac{\cs(P_n,w_n)}{w_n(P_n)}\leq \frac{1}{3} + \frac{2b}{an}.\]
Since $\frac{1}{3} + \frac{2b}{an}\to \frac{1}{3}$  as $n\to \infty$, this completes the proof.
\end{proof}}

\subsection{Graph suppression and $\G$}\label{subsec:subdivision}

\begin{proof}[Proof of Theorem~\ref{prop:subdivision}]
Let $G$ be a graph obtained by suppression of a vertex $v^*$ from a
graph $G^*$. Let $x$ and $y$ be the neighbors of $v^*$ in $G^*$.
It is sufficient to show that if $G^*\in \G,$ then $G\in \G$.
Assume $G^*\in \G$  and suppose $G\not\in\G$. Then there is a weight
function $w$ on $V(G)$ such that  $\s(G,w)<\cs(G,w)$.
Let $\epsilon$ be a {real number} such that
$0< \epsilon<\frac{\min\{\alpha,\beta\}}{2}$, where
\begin{eqnarray*}
&&\alpha= \min\{ \cs(G,w)-\s(G,w), w(x), w(y)\}, \\
&&\beta= \min_{T:\text{ non-safe set of }(G,w)}\{  w(D)-w(C) >0
\mid C\text{ is a component of }G[T], D\text{ is a component of }G-T\}.
\end{eqnarray*}
We define a weight function $w^*$ on $V(G^*)$ by
$w^*(v^*)=  \epsilon$, $w^*(x)=w(x)-\epsilon$
and $w^*({u})=w({u})$ for every ${u}\in V(G^*)\setminus \{x,v^*\}$
(see Figure~\ref{fig:proof:subdivision}).
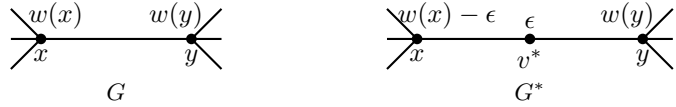
\begin{figure}[h!]\centering
    \begin{tikzpicture}[thick]
    \path
        (0,0) coordinate (x)
        (2,0) coordinate (y)
        (6.5,0) coordinate (v^*)
        (5,0) coordinate (x')
        (8,0) coordinate (y');
\fill  (1,-0.7) node  {\footnotesize$G$};
\fill  (6.5,-0.7) node  {\footnotesize$G^*$};
\fill
(x) node[below]{\small$x$}
(y) node[below]{\small$y$}
(x') node[below]{\small$x$}
(y') node[below]{\small$y$}
 (v^*) node[below]{\small$v^*$};
\fill   (x)+(0.2,0) node[above]{\small$w(x)$}
(y)+(-0.2,0) node[above]{\small$w(y)$}
(x')+(0.4,0) node[above]{\small$w(x)-\epsilon$}
(y')+(-0.2,0) node[above]{\small$w(y)$}
 (v^*) node[above]{\small$\epsilon$};
        \path (x) edge (y);
        \path (v^*) edge (x') edge (y');
        \path (y) edge (y)+(0.4,-0.4) edge (y)+(0.4,0) edge (y)+(0.4,0.4) edge (y);
        \path (y') edge (y')+(0.4,-0.4) edge (y')+(0.4,0) edge (y')+(0.4,0.4) edge (y');
        \path (x) edge (x)+(-0.4,-0.4) edge (x)+(-0.4,0) edge (x)+(-0.4,0.4) edge (x);
        \path (x') edge (x')+(-0.4,-0.4) edge (x')+(-0.4,0) edge (x')+(-0.4,0.4) edge (x');
    \fill
        (x) circle (2pt)
        (y) circle (2pt)
        (v^*) circle (2pt)
        (x') circle (2pt)
        (y') circle (2pt);
    \end{tikzpicture}
\caption{An illustration for the proof of Theorem~\ref{prop:subdivision}}\label{fig:proof:subdivision}
    \end{figure}
Since $G^*\in \G,$ there is a connected safe set $S^*$ of $G^*$
such that $w^*(S^*)=\s(G^*,w^*)$.
For any subset $X$ of $V(G)$, we define   a subset $\tilde{X}$ of $V(G^*)$ by
$$\tilde{X}=\left\{\begin{array}{ll}X\cup\{v^*\}&\text{if }x\in X,\\
X&\text{otherwise.}\end{array}\right.$$
Then $w^*(\tilde{X})=w(X)$ by  definition.

\begin{clm} $\s(G^*,w^*)\le  \s(G,w)$.
  \label{claim1}
\end{clm}

\begin{proof}[Proof of Claim~\ref{claim1}]
Let $U$ be a minimum safe set of $(G,w)$, that is, $w(U)=\s(G,w)$.
Then $w^*(\tilde{U})=w(U)$.
Note that the adjacency between the components of $G-U$ and $G[U]$ is
the same as the adjacency between the components of
$G^*-\tilde{U}$ and $\tilde{U}$. Thus, $\tilde{U}$ is a safe set for $(G^*,w^*),$
and $\s(G^*,w^*)   \le w^*(\tilde{U})=w(U) =\s(G,w)$.
\end{proof}

Let $T=S^*\setminus\{v^*\}$.
If $T$ is a connected safe set of $(G,w)$, then $\cs(G,w) \le w(T)$, and so by Claim~\ref{claim1},
\[ \s(G^*,w^*)\le  \s(G,w) < \cs(G,w) \le w(T)\le w^*(S^*)+\epsilon =\s(G^*,w^*)+\epsilon,\]
and so $\cs(G,w)-\s(G,w) {<} \epsilon<\alpha$, a contradiction to the definition of $\epsilon$.
Thus $T$ is not a connected safe set of $(G,w)$.
Since $G[T]$ is connected, $T$ is not a safe set of $(G,w)$.
Then there is a component $D$ of $G-T$ such that $E_G(D,T)\neq\emptyset$ and $w(D)>w(T)$.
We have  \begin{eqnarray}\label{eq:sub}
&&w^*(\tilde{D}) >  w^*(S^*)+\epsilon,
\end{eqnarray}
since
\[ w^*(\tilde{D}) = w(D) \ge w(T) +\beta >  w(T) +2\epsilon = w^*(\tilde{T})+2\epsilon
\ge w^*(S^*)+\epsilon,\]
where the first inequality is from the definition of $\epsilon$ and the last inequality is from
\begin{eqnarray}\label{claim4}
&&w^*(\tilde{T})+\epsilon \ge w^*(S^*).
\end{eqnarray}
We note that if $\tilde{T}=S^*$  or $T=S^*$, then \eqref{claim4} holds trivially.
If $\tilde{T}\neq S^*$ and $T\neq S^*$, then $v^*\in S^*$ and $x\not\in S^*$, which implies
$w^*(\tilde{T})=w^*(T)=w^*(S^*)-\epsilon$, and again \eqref{claim4} holds.
Also $D\subset G-T\subset G^*-S^*,$ and $D$ is connected,
hence so is $\tilde{D}.$

\begin{clm}\label{claim5}
$S^*\cap \{v^*,x,y\}=\{v^*,y\}$ and $D\cap\{x,y\}=\{x\}$.
\end{clm}

\begin{proof}[Proof of Claim~\ref{claim5}]
By \eqref{eq:sub} and the fact that $S^*$ is a connected safe set of $(G^*,w^*)$,
$\tilde{D}$ cannot be contained in a component of $G^*-S^*$.
Since $\tilde{D}$ is a connected subgraph of $G^*$, $\tilde{D}$ is not a subgraph
of $G^*-S^*$.
Since $D\subset G-T\subset G^*-S^*$, it follows that $v^*\in \tilde{D}$ and
$v^*\not\in G^*-S^*$. Thus {$v^*\in S^*$, which implies that $x\in D$.}
Then since $x\in D \subset G^*-S^*$, we have $x\not\in S^*$.
Since $G[S^*]$ is connected and $S^*\neq \{v^*\}$ by the definition of $\epsilon$,
{it follows that $y\in S^*$}. Thus $S^*\cap \{v^*,x,y\}=\{v^*,y\}$.
From $ D \subset  G^*-S^* $, we deduce $D\cap\{x,y\}=\{x\}$.
\end{proof}

{Since $D$ is a connected subgraph of $G^*-S^*$, Claim~\ref{claim5} implies that} $w^*(D) = w^*(\tilde{D})-\epsilon$. Together with~\eqref{eq:sub},
$w^*(D) = w^*(\tilde{D})-\epsilon > w^*(S^*)$, which  contradicts the fact that $S^*$ is a connected safe set of $(G^*,w^*)$.
\end{proof}

When $w$ is a weight function of $P_{2n+1}$ for some $2n+1<m$ defined
in Subsection~\ref{subsec:path}, then since $P_m$ is a subdivision of
$P_{2n+1}$, by defining $w^*$ as in the proof of Theorem~\ref{prop:subdivision},
we can obtain infinitely  many weight functions $w$ on $P_m$ satisfying
$w\not\in \W(P_m)$.

%%%%%%%%%%%%%%%%%%%%%%%%%%%%%%%%%%
\section{Proof of Theorem~\ref{thm:cycle}}\label{sec:cycle}
%%%%%%%%%%%%%%%%%%%%%%%%%%%%%%%%%%

\begin{proof}[Proof of Theorem~\ref{thm:cycle}] Suppose that there is
a cycle $C$ not in $\G$.
We take such a $C$ with the shortest length and a weight function $w$ on
$V(C)$ such that $\s(C,w)<\cs(C,w)$. {Note that $C$ has at least four vertices.}
Let $S$ be a minimum safe set
of $(C,w)$.
Let $X_1$, $X_3$, \ldots, $X_{2m-1}$ be the $m$ components of $G[S]$ and
$X_0,X_2,\ldots,X_{2m-2}$ be the $m$ components of $G-S,$
where the indices are considered as elements of $\Z_{2m}.$
We assume $E(X_i,X_{i+1})\neq\emptyset$ for each $i\in\Z_{2m}.$

\begin{clm}\label{size:one}
 For each $i\in\Z_{2m}$, $|X_i|=1$.
\end{clm}
\begin{proof}[Proof of Claim~\ref{size:one}]
Let $C^*$ be the graph such that $V(C^*)=\{X_1,\ldots,X_{2m}\}$ and
$E(C^*)= \{X_iX_{i+1}\mid i\in \Z_{2m} \}$.
Then we define a weight function $w^*$ on $V(C^*)$ by $w^*(X_i)= w(X_i)$  for each $i\in \Z_{2m}$.
Suppose that  $|X_i|\ge 2$ for some $i$, then $C^*$ is a cycle shorter than $C.$
So $C^*\in \G$ follows by minimality of $C,$ and hence $\s(C^*,w^*)=\cs(C^*,w^*).$
Since $S^*=\{X_1,X_3,\ldots,X_{2m-1}\}$ is a safe set of $(C^*,w^*)$, there is a connected safe set
$S_0^*$ of $(C^*,w^*)$ whose weight is at most $w^*(S^*)$.
Then $S_0=\cup_{X\in S_0^*} X$ is a connected safe set of $(C,w)$ and $w(S_0)=w^*(S_0^*)\le w^*(S^*)=w(S)$, which is a contradiction.
\end{proof}

By Claim~\ref{size:one}, we can assume $X_i=\{u_i\}$ for each $i\in \Z_{2m},$
so that $S=\{u_1,u_3,\ldots,u_{2m-1}\}$.
For simplicity, we let $V=V(C)$.
Let $\min(w)=\min\{w(u) ~|~ u \in V\}$ and $\max(w)=\max\{w(u) ~|~ u \in V\}$.
Then
\begin{eqnarray*}
&&\max(w) = \max\{w(u) ~|~ u \in S\}\quad \text{ and }\quad \min(w) = \min\{w(u) ~|~ u \in V \setminus S \}.
\end{eqnarray*}
Without loss of generality, we may assume $w(u_0)=\min(w)$.
Let $k$ be a nonnegative integer such that $k<m$ and
$w(u_{2k+1})=\max(w)$.
Note that {$w(u_{2i+1})\ge w(u_{2i})$  and $w(u_{2i-1})\ge w(u_{2i})$ for any $i\in\Z_m$ by the definition of a safe set, and so}
\begin{eqnarray*}
w(S)- w(V \setminus S) =\sum^{m-1}_{i=0} \big( w(u_{2i+1}) - w(u_{2i}) \big)  &\geq& \sum^{k}_{i=0} \big( w(u_{2i+1}) - w(u_{2i}) \big)\\&=& w(u_{2k+1}) + \sum^{k}_{i=1} \big( -w(u_{2i}) + w(u_{2i-1}) \big) - w(u_{0})\\&=&  w(u_{2k+1})- w(u_{0})=\max(w) - \min(w).\end{eqnarray*}
Hence,
\begin{eqnarray}\label{c2}
&&  w(S) - w(V  \setminus S) \ge  \max(w) - \min(w)
\end{eqnarray}
For every $i \in \Z_{2m}$, define $I_i=\{u_i, u_{i+1}, \ldots, u_{i+m-1}\}$.
Note that, for every $i \in \Z_{2m}$, at least one of the two sets $I_i$ and
$I_{i+m}=V \setminus I_i$ is a  safe set of $(C,w)$. Hence
 $w(I_r) \leq \frac{w(V)}{2} \le w(I_{r+1})$ for some $r \in \Z_{2m}$.
Then $w(I_{r+m+1}) \leq \frac{w(V)}{2} \le w(I_{r+m})$.
Thus both $I_{r+1}$ and $I_{r+m}$ are  safe sets of $(C,w)$, and so
$w(I_{r+1})- w(I_{r+m+1})\ge 0$ and $w(I_{r+m})- w(I_{r})\ge 0$.
Without loss of generality, we assume
\[ w(I_{r+1})-w(I_{r+m+1}) \le  w(I_{r+m})-w(I_{r}).\]
Since \begin{eqnarray*}
2 (w(I_{r+1})-w(I_{r+m+1})) &\leq&
(w(I_{r+1}) - w(I_{r+m+1})) + ( w(I_{r+m}) - w(I_{r}) ) \\
& = & (w(I_{r+1})- w(I_{r}) ) + ( w(I_{r+m})- w(I_{r+m+1}))\\
&= &  (w(u_{r+m}) - w(u_{r})) +  (w(u_{r+m}) - w(u_{r}))\\
& \le&2(\max(w) - \min(w)),
\end{eqnarray*}
it follows that
\begin{eqnarray}\label{c3}
&& w(I_{r+1})-w(I_{r+m+1})\le \max(w) - \min(w).
\end{eqnarray}
Then by \eqref{c2} and \eqref{c3},
\[ 2w(I_{r+1})-w(V)
=w(I_{r+1})-w(I_{r+m+1})  \le \max(w) - \min(w) \le w(S) - w(V\setminus S) = 2w(S) - w(V),\]
and hence
$w(I_{r+1}) \le w(S).$
Since the set $I_{r+1}$ is a connected weighted
safe set and $S$ is a  minimum safe set of $(C,w)$, we get a
contradiction to $\s(C,w)<\cs(C,w)$.
\end{proof}

\section{Proof of Theorem~\ref{prop:ratio}}\label{new}

We start with the following lemma:

\begin{lem}\label{lem:P3} Let $p$ and $q$ be positive integers, where $p\ge q.$
For a graph $G$,  if
$\s(G,w)\ge \frac{q}{p+q} w(G)$ for any weight function $w$ on $G$, then
$$\frac{k(G[S])}{k(G-S)} = \frac{q}{p}$$ for any $\emptyset \neq S\subsetneq V(G)$ such that $
\frac{k(G[S])}{k(G-S)} \le \frac{q}{p}.$
\end{lem}
\begin{proof}[Proof of Lemma~\ref{lem:P3}]
Let $S$ be a nonempty proper subset  of $V(G)$ such that {$\frac{k(G[S])}{k(G-S)} \le \frac{q}{p}$}.
Let $S_1, \ldots, S_t$ be the components of $G[S],$ where $t=k(G[S]),$ and
$U_1, \ldots, U_r$ be the components of $G-S,$ where $r=k(G-S),$
so that  $\frac{t}{r}\le \frac{q}{p}$.
Let $w$ be a weight function on $V(G)$ such that $w(S_i)=w(U_j)>0$ for
all $i$ and $j.$ Then $S$ is a safe set of $(G,w)$.
Thus $\s(G,w)\le   \frac{t}{t+r} w(G) $.
If $\frac{t}{r} < \frac{q}{p}$, then
\[ \frac{t}{r} < \frac{q}{p} \quad \Leftrightarrow \quad  \frac{r}{t} > \frac{p}{q} \quad \Leftrightarrow \quad \frac{t+r}{t}=1+\frac{r}{t} >  \frac{p}{q}+1= \frac{p+q}{q} \quad \Leftrightarrow \quad  \frac{t}{t+r}< \frac{q}{p+q},\]
which implies $\s(G,w)< \frac{q}{p+q}w(G)$, a contradiction.
Thus $\frac{t}{r}=\frac{q}{p}$.
\end{proof}

Now we prove Theorem~\ref{prop:ratio}.

\begin{proof}[Proof of Theorem~\ref{prop:ratio}] First we will show that (i), (ii), and (iii) are equivalent.
It is trivial that (i) implies (ii) by Theorem~\ref{thm:cycle}. By the case for $p=q=1$ of Lemma~\ref{lem:P3} and $S=\{u,v\}$ where $u$ and $v$ are not adjacent, it follows that (ii) implies (iii).
Suppose that (iii) is true.
Take a spanning tree $T$ of $G$ with a maximum diameter.
For any two pendent vertices $x$ and $y$ of $T$, if $xy\not\in E(G)$,
then $G-\{x,y\}$ is connected, a contradiction to (iii). Thus any two
pendent vertices of $T$ are adjacent in $G$. If there are at least
three pendent vertices, then we can obtain a spanning tree that has
greater diameter than $T$, a contradiction to the choice of $T$.
Thus, $T$ is a path and $G$ has a Hamiltonian cycle $C=v_1\cdots v_n$ {as $n\ge 5$}.
If $C$ has no chord, then $G$ is a cycle.
Suppose that $C$ has a chord. For any chord, say $v_1v_i$, of $C$,
if a vertex $x$ in $ \{v_2,\ldots, v_{i-1} \}$ and
a vertex  $y$ in $ \{v_{i+1},\ldots, v_{n} \}$ are not adjacent in $G$,
then $G-\{x,y\}$ is connected, a contradiction to (iii).
Thus, for any chord, say $v_1v_i$, of $C$,  a vertex $x$ in
$ \{v_2,\ldots, v_{i-1} \}$ and
a vertex $y$  in $ \{v_{i+1},\ldots, v_{n} \}$ are  adjacent in $G$.
Then such $xy$ becomes a chord of $C$.
Applying this argument again to the chords of $C$, we see that $G$ is a complete graph.
Hence, (iii) implies (i). Therefore, (i), (ii),
and (iii) are equivalent.

It is trivial that (i) implies (iv).
It is sufficient to show that (iv) implies (iii).
First, we give some definitions.
We denote a 2-subset $\{u,v\}$ by $uv$ even though it is not an edge, and in this case, we call it a non-edge.
For any graph $H$, let $E_c(H)$ {(resp. $E_d(H)$)} be the set of edges $uv$ such that $H-\{u,v\}$ is connected {(resp. disconnected)},
and let $N_c(H)$ {(resp. $N_d(H)$)} be the set of non-edges $uv$ such that $H-\{u,v\}$ is connected {(resp. disconnected)}.

Assume that $G$ satisfies (iv).
If $G$ has a cut vertex, then $q_{1,1}(G)=n$ and $q_{n-1,1}(G)\le n-1$,
a contradiction to the assumption. Thus $G$ is 2-connected.
Then from the definitions,
\[|E_c(G)|+|E_d(G)|=|E(G)|=q_{2,1}(G),\]
and by taking the complement in $V(G)$ of each  element of
$E_c(G) \cup N_c(G)$, \[|E_c(G)|+|N_c(G)|=q_{n-2,1}(G).\]
Then by (iv),
\begin{eqnarray}\label{eq:NE}
  |N_c(G)|=|E_d(G)|.
\end{eqnarray}
Suppose that we prove that $E_d(G)=\emptyset.$ Then $N_c(G)=\emptyset$ follows from \eqref{eq:NE};
that is, $ \{u,v\}\in N_d(G)$ holds for any two nonadjacent vertices $u$ and $v,$
and so  $G-\{u,v\}$ is disconnected. This implies (iii).
{Thus, in the following, we will finish the proof by showing $E_d(G)=\emptyset.$}

We will use the following basic property of 2-connected graphs.
\begin{itemize}
\item[($\sharp$)] If $H$ is 2-connected, then for any component
$D$ of $H-\{x,y\}$,
each of $D\cup\{x\}$ and $D\cup\{y\}$ is connected.
\end{itemize}
If there is a component $D$ of $H-\{x,y\}$ such that $x$ is not adjacent to any vertex of $D$, then $D$ is separated by the vertex $y$, and so $G-y$ is disconnected, a contradiction.
Similarly, if there is a component  $D$ of $H-\{x,y\}$ which is not adjacent to $y$, then $x$ is a cut vertex, a contradiction. Thus ($\sharp$) holds.

We also add some observations (O1) and (O2) on   $E_d(G)$.
For an edge $uv\in E_d(G)$,
\begin{itemize}
\item[\rm (O1)] every component $D$ of $G-\{u,v\}$ satisfies $1\le |D| \le |V(G)|-3$;
\item[\rm (O2)]  $N_c(G)  \supset \{ S\subset  (V(G)-\{u,v\}) \mid |S|=2, {|S\cap D |\le 1} \text{ for {any component} }D\text{ of }G-\{u,v\} \}$.
\end{itemize}
To see (O1), take any edge $uv\in E_d(G)$.
Note that $G-\{u,v\}$ is disconnected {and so} $1\le |D| \le |V(G)|-3$
follows. Also (O2) holds, {to see why, let $S=\{u',v'\}\subset (V(G)-\{u,v\})$ and
$|S\cap D|\le1$ for any component $D$ of $G-\{u,v\}$.
Then clearly $S$ is a non-edge.
From the fact that $G$ is 2-connected, together with {$(\sharp)$},
we can see that every component of $G-S$ contains one of $u$ and $v$, and so $G-S$ is connected.}

To show that \eqref{eq:NE} implies $E_d(G)=\emptyset,$
we prove its contrapositive, so we assume  $E_d(G)\neq \emptyset.$
{Then, at the end, we will reach a contradiction to \eqref{eq:NE} by
showing that $|N_c(G)|>|E_d(G)|$.
Since $E_d(G)\neq \emptyset,$} we can take an edge $u_1v_1 \in E_d(G)$ so that there is a
component $C_1$ of $G-\{u_1,v_1\}$  such that $u_1v_1$ is the unique edge of
$G[C_1\cup\{u_1,v_1\}]$ which belongs to $E_d(G)$.
We can take such an edge by considering all edges $uv\in E_d(G)$ and
all components $C$ of $G-\{u,v\},$ and choose $uv$ and $C$ so that $C$ is
as small as possible.
Let $G_1=G-C_1$ and let $N_1$ be the set of non-edges defined by
\[N_1=\{ xy \mid x\in C_1, y\in V(G)- (C_1\cup\{u_1,v_1\}) \} .\]

We proceed similarly to construct a maximal sequence of subgraphs
$G_0,G_1,\ldots,G_p$ of $G,$ where $G_0=G$ and $p\geq 1$ as follows.
Assume that we have  $u_{i-1}v_{i-1}$, $C_{i-1}$, $G_{i-1}$,
and $N_{i-1}$ for some $i\ge 2$.
As long as $E_d(G_{i-1})\setminus\{ u_1v_1,\ldots,
u_{i-1}v_{i-1}\}\neq \emptyset$, we continue recursively:
\begin{itemize}
\item[](Step 1) Take an edge $u_iv_i \in E_d(G_{i-1})\setminus
\{ u_1v_1,\ldots, u_{i-1}v_{i-1}\}$  so that there is a component
$C_i$ of $G_{i-1}-\{u_{i},v_{i}\}$ such that $u_{i}v_{i}$ is the
unique edge of $G_{i-1}[C_i\cup\{u_{i},v_{i}\}]$ which belongs to
$E_d(G_{i-1})\setminus\{ u_1v_1,\ldots, u_{i-1}v_{i-1}\}$;
\item[](Step 2) Let $G_i=G_{i-1}-C_i$ and let $N_{i}=\{ xy \mid
x\in C_{i}, y\in V(G_{i-1})- (C_{i}\cup\{u_{i},v_{i}\}) \}$.
\end{itemize}
{We note that (Step 1) is possible by choosing the edge $u_iv_i \in
E_d(G_{i-1})\setminus \{ u_1v_1,\ldots, u_{i-1}v_{i-1}\}$
and the component $C_i$ of $G_{i-1}-\{u_{i},v_{i}\}$
so that $|C_i|$ is as small as possible.
If the subgraph induced by
$C_i\cup\{ u_i,v_i\}$ contains an edge $u'v'$ in the set
$E_d(G_{i-1})\setminus \{ u_1v_1,\ldots, u_{i-1}v_{i-1}\}$,
then $G_{i-1}-\{u',v'\}$ has a component whose order is smaller
than $|C_i|$, a contradiction to the choice of $u_iv_i$. Thus
$u_{i}v_{i}$ is the unique edge of
$G_{i-1}[C_i\cup\{u_{i},v_{i}\}]$ which belongs to
$E_d(G_{i-1})\setminus\{ u_1v_1,\ldots, u_{i-1}v_{i-1}\}$.
}

\begin{clm}\label{claim:2connected}{
$G_i$ is $2$-connected for each $i\le p$.
}
\end{clm}
\begin{proof}[Proof of Claim~\ref{claim:2connected}]
For $G_0=G$ this is true by assumption.
Suppose that $G_i$ has a cut vertex $x$ for some $i\ge1$, where $G_j$ is 2-connected for all $j<i$.
Since both $u_i$ and $v_i$ are vertices of $G_i$, we may assume
that $u_i$ is a vertex of $G_i-x$.
Let $C$ be the component of $G_i-x$ which contains $u_i$, and $C'$ be another component of $G_i-x$.
Since $u_iv_i\in E(G_i)$, note that if $v_i\neq x$, then $v_i\in C$.
Recall that $G_i=G_{i-1}-C_{i}$.
By the minimality of $i$, $G_{i-1}$ is 2-connected and so by ($\sharp$), both $C_{i}\cup\{u_i\}$ and $C_{i}\cup\{v_i\}$ are connected.
Since $C_{i}$ is a component of $G_{i-1}-\{u_i,v_i\}$, $C_{i}$ is
adjacent to only $\{u_i,v_i\}$ among all vertices of $G_{i-1}$.
Hence, $C\cup C_{i}$ is connected in $G_{i-1}$ and $C_i$ is adjacent
to only $C$ among the components of $G_i-x$. This implies that $C'$
is still a component of $G_{i-1}-x$, which implies that $G_{i-1}-x$
is disconnected, a contradiction.
\end{proof}

\begin{clm}\label{claim:connected}
For any $\{u,v\}\subset V(G_{i}),$ if $\{u,v\}\neq \{u_i,v_i\},$ then
\begin{itemize}
\item[(a)] if $G_i-\{u,v\}$ is connected then $G_{i-1}-\{u,v\}$ is connected,
\item[(b)] if $uv\in E_d(G_i)$, then $G_{i-1}-\{u,v\}$ is disconnected.
\end{itemize}
\end{clm}

\begin{proof}[Proof of Claim~\ref{claim:connected}]
Take $\{u,v\}\subset V(G_{i})$ so that $\{u,v\}\neq \{u_i,v_i\}$.
Since $\{u,v\}\neq \{u_i,v_i\}$,
we may assume that $u_i\notin \{u,v\} $.
Let $C$ be the component of $G_i-\{u,v\}$ containing the vertex $u_i$.
Recall that $C_i$ is a component of $G_{i-1}-\{u_i,v_i\}$ taken from
(Step 1), and by {Claim~\ref{claim:2connected} and ($\sharp$)},
$\{u_i\}\cup C_i$ induces a connected graph in $G_{i-1}$.
Each of $\{u_i\}\cup C_i$ and $C$ is a connected graph in
$G_{i-1}$ containing the vertex $u_i$.
Hence, $\{u_i\}\cup C_i\cup C=C_i\cup C$ induces a connected graph in $G_{i-1}$.
Since each of $C_i$ and $C$ is disjoint from $\{u,v\}$, $C_i\cup C$
induces a connected graph $H$ in $G_{i-1}-\{u,v\}$.

To show (a), suppose that $G_i-\{u,v\}$ is connected.
Then $C=G_{i}-\{u,v\}$ and so $H$ is a connected spanning subgraph of $G_{i-1}-\{u,v\}$.
Thus  $G_{i-1}-\{u,v\}$  is a connected graph, and (a) holds.
To show (b), suppose that $uv\in E_d(G_i)$.
Then $G_{i}-\{u,v\}$ is disconnected. From {the fact that $u_i$ and $v_i$ are adjacent}, $v_i \in C$ if $v_i\not\in \{u,v\}$. Note that $C_i$ is only connected to two vertices $u_i$ and $v_i$ among all vertices of $G_{i-1}$. Thus, $C$ is the unique component of $G_i-\{u,v\}$ which is adjacent to $C_i$.
Hence, $G_{i-1}-\{u,v\}$ is not connected, and so $uv\in  E_d(G_{i-1})$.
\end{proof}

\begin{clm}\label{claim:finish}
For every $i=1,\ldots,p,$
\[E_d(G_{i})\setminus\{ u_1v_1,\ldots, u_{i}v_{i} \}
=E_d(G_{i-1})\setminus\{ u_1v_1,\ldots, u_{i}v_{i} \}.\]
\end{clm}
\begin{proof}[Proof of Claim~\ref{claim:finish}]
For simplicity, let $E_i=E_d(G_{i})\setminus\{ u_1v_1,\ldots, u_{i}v_{i} \}$
and $E_{i-1}=E_d(G_{i-1})\setminus\{ u_1v_1,\ldots, u_{i}v_{i} \}$.
 Take an edge $uv \in E_i$. By (b) of Claim~\ref{claim:connected}, $G_{i-1}-\{u,v\}$ is disconnected, and so $uv\in  E_d(G_{i-1})$.
Since $uv\not\in \{ u_1v_1,\ldots, u_{i}v_{i} \}$, $uv\in E_{i-1}$.
Thus $E_i \subset E_{i-1}$.
To show that $E_{i-1}  \subset E_i$, take
 an edge $uv \in E_{i-1}$.
By the definition of $C_i$, for any edge $u'v'$ in $G_{i-1}[C_{i}\cup\{u_i,v_i\}]$ except {the edges in $\{u_1v_1,\ldots, u_iv_i\}$},
$G_{i-1}-\{u',v'\}$ is connected. Therefore, from the fact that  $uv \in E_d(G_{i-1})$, $uv$ is an edge of $G_i$.
By (a) of Claim~\ref{claim:connected}, $G_i-\{u,v\}$ is disconnected, and so $uv\in  E_d(G_i)$.
Since $uv\not\in \{ u_1v_1,\ldots, u_{i}v_{i} \}$, $uv\in E_{i}$.
Hence, the claim holds.
\end{proof}

\begin{clm}\label{claim:last}
$N_i\neq \emptyset,$ $N_i \cap N_c(G_i)=\emptyset,$ and
$N_i \cup N_c(G_i) {\subset} N_c(G_{i-1})$ for every
$i=1,\ldots,p.$  Moreover,  $|N_1|\ge 2$.
\end{clm}

\begin{proof}[Proof of Claim~\ref{claim:last}]
By the definition of $N_i$ and $G_i$, it is clear that
$N_i\neq \emptyset$ and $N_i \cap N_c(G_i)=\emptyset$ hold.
By (O2), $N_i$ is a subset of $N_c(G_{i-1})$.
For any non-edge $xy\in N_c(G_i)$,  $G_i-\{x,y\}$ is connected
and so $G_{i-1}-\{x,y\}$ is a connected graph  by (a) of
Claim~\ref{claim:connected}, which implies
$xy\in N_c(G_{i-1})$. Thus $N_c(G_i)\subset N_c(G_{i-1})$.

Moreover, from (O1), by the assumption of $n\ge 5,$ we have
$|N_1|\ge |C_1| (n-2-|C_1|)\ge n-3 \ge 2$.
\end{proof}

From Claim~\ref{claim:finish}, for each $i=1,\ldots,p,$
\begin{eqnarray*}
E_d(G_{i})\setminus\{ u_1v_1,\ldots, u_{i}v_{i} \}
=E_d(G_{0})\setminus\{ u_1v_1,\ldots, u_{i}v_{i} \},
\end{eqnarray*}
which implies $p=|E_d(G)|,$ since
$E_d(G_{i})\setminus\{ u_1v_1,\ldots, u_{i}v_{i} \}\neq \emptyset$
for $i<p,$ and
$E_d(G_{p})\setminus\{ u_1v_1,\ldots, u_{p}v_{p} \} =\emptyset.$
Thus by Claim~\ref{claim:last},
\begin{eqnarray*}
&& N_c(G)=N_c(G_0) \ \supset \ N_1 \cup N_c(G_1) \ \supset  \ N_1 \cup N_2 \cup N_c(G_2) \ \supset \ \cdots \ \supset \ N_1 \cup N_2 \cup \cdots \cup N_p,
\end{eqnarray*}
where $N_i$ and $N_j$ are disjoint whenever $1 \leq i < j \leq p.$
Again by Claim~\ref{claim:last},
\[  |N_c(G)| \ge |N_1| +|N_2|+ \cdots + |N_p| \ge 2 + \underbrace{1 + \cdots +1}_{ (p-1)\text{ times}} =p+1 > p =|E_d(G)|, \]
which violates \eqref{eq:NE}. Hence the proof is complete.
\end{proof}

\section{Closing remarks}\label{sec:open}

We finally give two remarks in our main results.

\begin{rmk}
From
the proof of
Theorem~\ref{thm:cycle}, we have the following
strongly linear time
algorithm for calculating the safe number of a cycle with a
weight function.

\noindent
\hrulefill \\
\textsc{{\bf WEIGHTED SAFE NUMBER OF A CYCLE GRAPH}}\\
\textsc{{\bf INPUT}}: A cycle {$C$} such that
$V(C):=\{v_i | i \in \Z_n \}$ and $E(C):=\{v_iv_{i+1}| i \in \Z_n\}$
and a {positive real-valued function $w$ on $V(C)$}. \\
\textsc{{\bf OUTPUT}}: The {(connected)} safe number $\s(C,w)${$(=\cs(C,w))$}.
\begin{quote}
\begin{alg-enumerate}
\item{Calculate the total weight $w(V):=\sum_{i=0}^{n-1} w(v_i)$.}
\item{Set $w_{\min}:=w(V)$}
\item{Set $w:=w(v_0)$}
\item{Set $\ell:=0 \/ (\in \Z_n)$}
\item{Set $k:=0 \/ (\in \Z_n)$}
\item{While $w < \frac{w(V)}{2}$ do:\\
    \hspace{30pt} set $w := w + w(v_{\ell+1});$ \hspace{2pt} set $\ell:=\ell+1;$}\label{6}
\item{If $w < w_{\min}$ then set $w_{\min}:=w$}
\item{Set $w:=w - w(v_k)$}
\item{Set $k:=k+1$}
\item{If $k=0$ then return the number $w_{\min}$}
\item{Goto Step\hspace*{0.2em}\ref{6}}
\end{alg-enumerate}
\end{quote}
\hrulefill

{We remark that for each $k\in\Z_n$, Step 6 determines the `smallest' $\ell\in\Z_n$ such that $\{v_k,v_{k+1},\ldots,v_\ell\}$ has weight at least $\frac{w(V)}{2}$.}

\end{rmk}

Note that, in contrast with the above, it was shown in
\cite{BFLMMST-weighted-sf} that the safe number of
a given weighted path can be calculated in $O(n^3)$-time.

\begin{rmk} We note that (i), (ii), and (iii) in Theorem~\ref{prop:ratio}
are equivalent without the assumption of $n \ge 5$.
In addition, if we replace (iii) and/or (iv) in Theorem~\ref{prop:ratio} by
any of the following stronger statements, then the theorem remains true:
\begin{itemize}
  \item[(iii${}^{\prime}$)] $k(G-S)=k(G[S])$ for any $S\subset V(G)$ with $|S|=2$;
  \item[(iii${}^{\prime\prime}$)] $k(G-S)=k(G[S])$ for any $S\subset V(G)$ with $S\neq \emptyset$.
  \item[(iv${}^{\prime}$)] $q_{k,1}(G)=q_{n-k,1}(G)$ for any $1 \le k \le n-1$.
\end{itemize}
Obviously, (i) implies (iii${}^{\prime}$), (iii${}^{\prime\prime}$),  and (iv${}^{\prime}$).
Each of (iii${}^{\prime}$) and (iii${}^{\prime\prime}$) implies (iii),
and  (iv${}^{\prime}$) implies (iv).
\end{rmk}

\section*{Acknowledgement}
The authors would like to thank the two anonymous reviewers for helpful and valuable comments.

\end{document}